\documentclass{amsart}

\usepackage{amssymb}
\usepackage{bbm}
\usepackage{amsaddr}

\newtheorem{thm}{Theorem}[section]
\newtheorem{prop}[thm]{Proposition}
\newtheorem{lem}[thm]{Lemma}
\newtheorem{cor}[thm]{Corollary}
\newtheorem{conj}[thm]{Conjecture}

\theoremstyle{definition}

\theoremstyle{remark}
\newtheorem{rem}[thm]{Remark}

\numberwithin{equation}{section}

\newcommand{\f}{\frac}
\newcommand{\lt}{\left}
\newcommand{\rt}{\right}

\newcommand{\vocab}{\textit}

\newcommand{\mean}{\overline}
\newcommand{\functo}{\rightarrow}
\newcommand{\approach}{~\rightarrow~}

\newcommand{\EE}{\mathbb{E}}
\newcommand{\GG}{\mathbb{G}}
\newcommand{\NN}{\mathbb{N}}
\newcommand{\PP}{\mathbb{P}}
\newcommand{\one}{\mathbb{1}}

\newcommand{\cC}{\mathcal{C}}
\newcommand{\cH}{\mathcal{H}}
\newcommand{\cM}{{\mathcal{M}}}
\newcommand{\cW}{{\mathcal{W}}}

\newcommand{\ff}{{\mathfrak{f}}}
\newcommand{\fg}{{\mathfrak{g}}}
\newcommand{\fm}{{\mathfrak{m}}}

\newcommand{\eps}{{\varepsilon}}

\makeatletter
\renewcommand{\email}[2][]{%
  \ifx\emails\@empty\relax\else{\g@addto@macro\emails{,\space}}\fi%
  \@ifnotempty{#1}{\g@addto@macro\emails{\textrm{(#1)}\space}}%
  \g@addto@macro\emails{#2}%
}
\makeatother

\begin{document}

\title{Convergence of Maximum Bisection Ratio of Sparse Random Graphs}

\author{Brice Huang}

\address{Department of Mathematics, MIT, Cambridge, MA, USA}
\email{bmhuang@mit.edu}

\begin{abstract}
  We consider sequences of large sparse random graphs
  whose degree distribution approaches a limit with finite mean.
  This model includes both the random regular graphs
  and the Erd\"os-Renyi graphs of constant average degree.
  We prove that the maximum bisection ratio of such a graph sequence
  converges almost surely to a deterministic limit.
  We extend this result to so-called 2-spin spin glasses
  in the paramagnetic to ferromagnetic regime.
  Our work generalizes the graph interpolation method
  to some non-additive graph parameters.
\end{abstract}

\maketitle

\section{Introduction}\label{sec-intro}

The \vocab{interpolation method} is used
in a remarkable paper by Guerra and Toninelli \cite{GuTo}
to prove the existence of an infinite volume limit of thermodynamic quantities.
In this method, a system of size $n$ is compared,
by a sequence of interpolating systems,
to a pair of independent systems of size $n_1$ and $n_2$, where $n_1+n_2=n$.
If, at each step of the interpolation, the parameter of interest increases,
then the parameter is subadditive in $n$,
and therefore converges when divided by $n$.

Bayati, Gamarnik, and Tetali \cite{BGT} adapted this technique
in a combinatorial setting as \vocab{graph interpolation}.
Using graph interpolation, \cite{BGT} proved that
in both the sparse Erd\"os-Renyi and $d$-regular random graph models,
several graph parameters, including independence number and maxcut size,
converge when divided by $n$.
Gamarnik \cite{Gam} showed an analogous result
for log-partition functions in the context of right-convergence of graphs,
and found that the subadditivity required for graph interpolation
follows from a concavity property of the graph parameter.

In a recent synthesis, Salez \cite{Sal} further generalized these results
by identifying the properties of these parameters that permit interpolation;
Salez proved that an interpolation argument succeeds whenever
the graph parameter satisfies additivity, Lipschitz, and concavity conditions.
Moreover, \cite{Sal} generalized
the $d$-regular random graph model of \cite{BGT}
to graphs with arbitrary degree distribution generated by a configuration model.

The interpolation arguments in the literature all depend
on an additivity property of the graph parameter --
that if $G$ is the vertex-disjoint union of graphs $G_1, G_2$,
the graph parameter $f$ satisfies $f(G)=f(G_1)+f(G_2)$.
While many graph parameters of interest, such as independence number,
maxcut, $K$-SAT, and log-partition functions all have this property,
other graph parameters, such as maximum bisection, do not.

In this paper, we show that the maximum bisection parameter
in the arbitrary degree sequence model converges when divided by $n$.
The random regular graph case of our result resolves
an open problem on spin glasses \cite[Problem 2.3]{Aim}.
The analogue of this problem for Erd\"os-Renyi random graphs
was resolved in an unpublished result of Gamarnik and Tetali;
this result is also implied by our result.

We then consider a type of $p$-hybrid bisections for $p\in [0,1]$,
interpolating between the maximum and minimum bisections.
These are the maximum bisections of the ``2-spin spin glass" model
studied by Franz and Leone in \cite{FrLe},
where the parameter $p$ determines the ferromagnetism of the system.
We show that for $p\ge \f12$, the $p$-hybrid bisection
in the arbitrary degree sequence model also converges when divided by $n$.
In other words, the maximum bisection of the 2-spin spin glass model
has a scaling limit in the paramagnetic to ferromagnetic regime.

The key idea allowing us to extend the results in \cite{BGT} and \cite{Sal}
to maximum bisection and maximum $p$-hybrid bisection,
which are not additive, is to consider $(A,B)$-bisections,
bisections that also bisect two given sets $A,B$ that partition $V(G)$.
This added constraint allows us to decompose a system into two parts,
an operation that previously depended on additivity.
By showing a form of subadditivity on maximum $(A,B)$-bisections,
we can show subadditivity on maximum bisection
and establish the existence of a scaling limit.

\subsection*{Acknowledgements}
The author gratefully acknowledges Mustazee Rahman
for many insightful conversations,
and for bringing much of the relevant literature to the author's attention.
The author also thanks the MIT Math Department's
Undergraduate Research Opportunities Program,
in which this work was completed.

\section{Preliminaries}\label{sec-prelim}

\subsection{Random Graphs with Given Degree Sequence}

Throughout this paper, we will work with finite, undirected graphs,
where loops and multiple edges are allowed.

We will work with the following random graph model.
Consider nodes $[n] = \{1,\dots,n\}$, and a degree function $d:[n] \functo \NN$.
Create a multiset $\cH_d$ of nodes, where each $i\in [n]$ appears $d(i)$ times.
Each (possibly partial) matching $\fm$ of $\cH_d$ induces a graph $G[\fm]$,
which contains an edge $(i,j)$ for every pair $\{i,j\} \in \fm$.
Note that if $\fm$ is not a complete matching,
some vertex $i$ in $G[\fm]$ will have degree less than $d(i)$.

\begin{figure}[h]
  \centering
  \begin{picture} (100, 83.33)
    \put(00.0,41.67){\circle*{2}}
    \put(25.0,00.00){\circle*{2}}
    \put(25.0,83.33){\circle*{2}}
    \put(75.0,00.00){\circle*{2}}
    \put(75.0,83.33){\circle*{2}}
    \put(100.0,41.67){\circle*{2}}

    \put(-7.0,38.67){$1$}
    \put(18.0,-3.00){$3$}
    \put(18.0,80.3){$1$}
    \put(77.0,-3.00){$4$}
    \put(77.0,80.33){$2$}
    \put(102.0,38.67){$2$}

    \put(-3.0,41.67){ \line(3,-5){25} }
    \put(97.0,41.67){ \line(-3,-5){25} }
    \put(22.0,83.33){ \line(1,0){50}}
  \end{picture}
  \caption{A matching of $\cH_d$, where $d(1)=d(2)=2$ and $d(3)=d(4)=1$.}
\end{figure}
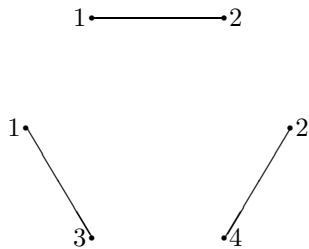

We let $\GG_{d}$ denote the distribution of $G[\fm]$,
where $\fm$ is a uniformly random complete matching on $\cH_d$.
Note that when $d$ is a constant function with value $r$,
$\GG_d$ is the random $r$-regular graph model,
and when $d$ is sampled from the degree distribution
of an Erd\"os-Renyi random graph,
the doubly-random $\GG_d$ is the corresponding Erd\"os-Renyi random graph model.

We say a sequence $\{d_n: [n]\functo \NN \}_{n\ge 1}$
\vocab{converges in distribution} to a probability measure
$\mu: \NN\functo [0,1]$ with finite mean $\mean{\mu}$ if for all $k\in \NN$,
\begin{equation}
  \f{1}{n} \sum_{i\in [n]} \one_{d_n(i) = k} \approach \mu(k)
\end{equation}
and
\begin{equation}
  \f{1}{n} \sum_{i\in [n]} d_n(i) \approach \mean{\mu}
\end{equation}
as $n\approach \infty$.

The results in this paper are concerned with
families of random graphs $\{\GG_{d_n}\}_{n\ge 1}$,
where each $\GG_{d_n}$ is sampled independently,
and where the degree functions $d_n$
converge in distribution to a measure $\mu$ with finite mean.

\subsection{Graph Parameters}

A \vocab{graph parameter} is a real-valued,
isomorphism-invariant function on graphs.
Given a graph parameter $\ff$ and a graph $G$,
define $\Delta^{G, \ff}$ as the matrix given by
\begin{equation}
  \Delta^{G, \ff}_{ij} = \ff(G+ij)-\ff(G)
\end{equation}
for $i,j\in V(G)$.

We say a graph parameter $\ff$ is \vocab{additive} if
\begin{equation}
  \ff(G)=\ff(G_1)+\ff(G_2)
\end{equation}
when $G$ is the disjoint union of $G_1$ and $G_2$.
We say $\ff$ is \vocab{1-Lipschitz} if for all $G$, and all $i,j\in V(G)$,
\begin{equation}
  |\Delta^{G,\ff}_{ij}| \le 1.
\end{equation}
Finally we say $\ff$ is \vocab{concave} if
\begin{equation}
  x\cdot \one = 0 \Rightarrow x^T \Delta^{G,\ff} x \le 0,
\end{equation}
where $\one$ is the all-1 vector on $V(G)$.

The most general result on graph parameters is due to Salez.
\begin{thm}\label{thm-sal}\cite{Sal}
  Let $\ff$ be an additive, $\kappa$-Lipschitz, concave graph parameter,
  and let $\{d_n\}_{n\ge 1}$ converge in distribution
  to a measure $\mu$ with finite mean.
  Then, the sequence of independent samples
  \begin{equation}
    \f{1}{n} \ff\lt(\GG_{d_n}\rt)
  \end{equation}
  converges almost surely to a limit $\Psi(\mu)$ as $n\approach \infty$.
  Moreover, the scaling limit
  \begin{equation}\label{eq-salez-scaling-limit}
    \lim_{n\approach \infty} \f{1}{n} \EE\lt[\ff\lt(\GG_{d_n}\rt)\rt]
  \end{equation}
  exists and equals $\Psi(\mu)$.
\end{thm}

Bayati, Gamarnik, and Tetali \cite{BGT} showed, before Salez,
that the scaling limit (\ref{eq-salez-scaling-limit}) exists for the
max-cut, independence number, K-SAT, and not-all-equal K-SAT parameters,
in the random $r$-regular graph and Erd\"os-Renyi random graph models.
As these parameters all satisfy the hypothesis of Theorem~\ref{thm-sal},
and the random regular graph and Erd\"os-Renyi random graph models
are special cases of the arbitrary degree sequence model,
this result is a consequence of Theorem~\ref{thm-sal}.

\section{Results}\label{sec-results}

\subsection{Graph Bisections}

Define the \vocab{maximum bisection} of a graph $G$ by
\begin{equation}
  MB(G) =
  \max \lt\{
    e(V_1,V_2) |
    \text{$V_1,V_2$ partition $V(G)$, $\Big||V_1|-|V_2|\Big| \le 1$}
  \rt\},
\end{equation}
where $e(V_1,V_2)$ is the number of edges between $V_1$ and $V_2$ in $G$.
Observe that the maximum bisection is not additive,
and therefore Theorem~\ref{thm-sal} does not apply.

The first result of this paper is:
\begin{thm}\label{thm-maxbi}
  Let $\{d_n\}_{n\ge 1}$ converge in distribution
  to a measure $\mu$ with finite mean.
  Then,
  \begin{equation}\label{eq-maxcut-almost-sure-conv}
    \f{1}{n} MB\lt(\GG_{d_n}\rt),
  \end{equation}
  where each $\GG_{d_n}$ is sampled independently,
  converges almost surely as $n\approach \infty$.
  Moreover, the scaling limit
  \begin{equation}\label{eq-maxcut-scaling-limit}
    \lim_{n\approach \infty} \f{1}{n} \EE\lt[MB\lt(\GG_{d_n}\rt)\rt]
  \end{equation}
  exists.
\end{thm}

Whether the same result holds
for the minimum bisection graph parameter is an open problem.
In fact, the random regular graph case of this problem
is implied by the following stronger conjecture.
\begin{conj}\cite{ZdBo}\label{conj-lenka}
  Let $MC$ and $mB$ denote, respectively,
  the max-cut and min-bisection parameters,
  and let $\GG(n,r)$ be a random $r$-regular graph on $n$ vertices.
  Then,
  \begin{equation}\label{eq-lenka}
    MC(\GG(n,r)) + mB(\GG(n,r)) = |E| + o(n),
  \end{equation}
  where $|E|=\f12 nr$ is the number of edges in $\GG(n,r)$.
\end{conj}

\subsection{Hybrid Bisections}

We define the \vocab{$p$-hybrid bisection} $HB_p$ of a graph $G$ as follows.
Let $\Omega$ be a labeling of the edges of $G$,
with each edge independently labeled $+1$ with probability $p$,
and $-1$ with probability $1-p$,
and let $G(\Omega)$ denote $G$ with the labeling $\Omega$.
In the statistical physics literature
(cf. \cite{FrLe}, \cite{FLT}, \cite{PaTa}),
the graph $G(\Omega)$ is a \vocab{2-spin spin glass},
with the parameter $p$ determining the system's magnetism:
the system is ferromagnetic at $p=1$, paramagnetic at $p=\f12$,
and antiferromagnetic at $p=0$.

We define
\begin{equation}
  HB_p(G) = \EE\lt[MB(G(\Omega))\rt],
\end{equation}
where the expectation is over the randomness of $\Omega$.

Note that when $p=1$, a $p$-hybrid bisection is a max bisection,
and when $p=0$, a $p$-hybrid bisection is a min bisection.
Our main result is:
\begin{thm}\label{thm-hybrid}
  Fix $p\ge \f12$,
  and let $\{d_n\}_{n\ge 1}$ converge in distribution
  to a measure $\mu$ with finite mean.
  Then,
  \begin{equation}\label{eq-hybrid-almost-sure-conv}
    \f{1}{n} HB_p\lt(\GG_{d_n}\rt),
  \end{equation}
  where each $\GG_{d_n}$ is sampled independently,
  converges almost surely as $n\approach \infty$.
  Moreover, the scaling limit
  \begin{equation}\label{eq-hybrid-scaling-limit}
    \lim_{n\approach \infty} \f{1}{n} \EE\lt[HB_p\lt(\GG_{d_n}\rt)\rt]
  \end{equation}
  exists.
\end{thm}

As Theorem~\ref{thm-maxbi} is a special case of Theorem~\ref{thm-hybrid},
the rest of this paper will be devoted to proving Theorem~\ref{thm-hybrid}.

\begin{rem}
  Let $\alpha \in (0,1)$.
  We can define an \vocab{$\alpha$-cut} of $G$ as a
  cut that partitions $V(G)$ into the ratio $\alpha:(1-\alpha)$.
  In particular, a bisection is a $\f12$-cut.
  Theorems~\ref{thm-maxbi} and \ref{thm-hybrid} remain true
  when ``bisection" is replaced by ``$\alpha$-cut,"
  and their proofs are analogous.
\end{rem}

\section{Graph Pseudo-Parameters}

\subsection{Constrained Max-Bisections}

Define a \vocab{graph pseudo-parameter} as a real-valued,
not necessarily isomorphism-invariant function on graphs.

The main idea that allows us to consider
the non-additive parameter $HB_p$ is as follows.
Let $A,B$ be a partition of the vertices of a graph $G$.
Say an \vocab{$(A,B)$-bisection} is a bisection of the vertices of $G$
that also bisects the sets $A,B$.
Let $MB^{A,B}$ denote the maximum $(A,B)$-bisection of $G$.
Analogously, define
\begin{equation}
  HB_p^{A,B}(G) = \EE\lt[MB^{A,B}(G(\Omega))\rt],
\end{equation}
where $\Omega$ is defined as before.
Note that both $MB^{A,B}$ and $HB_p^{A,B}$ are graph pseudo-parameters,
and that they are additive in the following sense.
When there are no edges from $A$ and $B$,
\begin{equation}
  MB^{A,B}(G) = MB(G[A])+MB(G[B])
\end{equation}
and
\begin{equation}
  HB_p^{A,B}(G) = HB_p(G[A])+HB_p(G[B])
\end{equation}
where $G[A]$ and $G[B]$ are the induced subgraphs of $G$ on $A$ and $B$.

We will prove the following result, which, in light of the bound
\begin{equation}
  HB_p^{A,B}(G) \le HB_p(G),
\end{equation}
will imply that $HB_p$ is subadditive.
\begin{prop}\label{prop-subadd}
  Let $A,B$ be a partition of $[n]$.
  Let $d: [n]\functo \NN$ be a function,
  and let $d\uparrow A, d\uparrow B$ denote its restrictions to $A$ and $B$.
  Then,
  \begin{equation}
    \EE\lt[HB_p(\GG_{d\uparrow A})\rt] +
    \EE\lt[HB_p(\GG_{d\uparrow B})\rt]
    \le
    \EE\lt[HB_p^{A,B}(\GG_d)\rt] +
    \psi\lt(|E(\GG_d)|\rt),
  \end{equation}
  where $\psi(x) = 7\sqrt{x\log (1+x)}$.
\end{prop}

\subsection{Graph Interpolation}

Fix a partition $A,B$ of $[n]$ and a function $d: [n]\functo \NN$.
Say that an edge in a matching $\fm \in \cH_d$ is an \vocab{$A$-edge}
if both of its endpoints are in $A$; define a \vocab{$B$-edge} analogously.
Say an edge is a \vocab{cross-edge}
if it has an endpoint in each of $A$ and $B$.

For $\alpha,\beta,\gamma\in \NN$,
define $\cM(\alpha,\beta,\gamma)$ to be the set of matchings $\fm\in \cH_d$ with
$\alpha$ $A$-edges, $\beta$ $B$-edges, and $\gamma$ cross-edges;
say $(\alpha, \beta, \gamma)$ is \vocab{feasible}
if at least one such matching exists.

For a graph pseudo-parameter $\fg$ and feasible $(\alpha,\beta,\gamma)$, define
\begin{equation}
  F_{\fg}(\alpha,\beta,\gamma)
  =
  \EE_{\fm \in \cM(\alpha,\beta,\gamma)} \fg\lt(G[\fm]\rt).
\end{equation}
Let
\begin{equation}
  d(A) := \sum_{i\in A} d(i),
  \quad
  d(B) := \sum_{i\in B} d(i)
\end{equation}
be the total degree of the sets $A$ and $B$.
The following result is the proof of Theorem 3 of \cite{Sal};
while this was proved for graph parameters,
its proof extends directly to pseudo-parameters.

\begin{prop}\cite{Sal}\label{prop-graph-interpolation}
  Let $\fg$ be a graph pseudo-parameter obeying the following conditions.
  \begin{itemize}
    \item (Lipschitz Continuity)
    For any feasible $(\alpha,\beta,\gamma)$ and $(\alpha',\beta',\gamma')$:
    \begin{equation}\label{eq-def-lipschitz-continuity}
      |F_{\fg}(\alpha,\beta,\gamma)-F_{\fg}(\alpha',\beta',\gamma')|
      \le
      |\alpha-\alpha'|+|\beta-\beta'|+|\gamma-\gamma'|.
    \end{equation}
    \item (Local Super-Additivity)
    If $\delta \ge 2$ and $(\alpha,\beta,\gamma+\delta)$ is feasible, then
    \begin{equation}\label{eq-def-local-super-additivity}
      \f12
      \lt(
        F_{\fg}(\alpha+1,\beta,\gamma) + F_{\fg}(\alpha,\beta+1,\gamma)
      \rt)
      \le
      F_{\fg}(\alpha,\beta,\gamma+1) + \f{2}{\delta}.
    \end{equation}
  \end{itemize}
  Then, for any $\gamma \le \min(d(A),d(B))$,
  \begin{equation}\label{eq-graph-interpolation-result}
    F_{\fg}
    \lt(
      \lt\lfloor \f{d(A)}{2}\rt\rfloor,
      \lt\lfloor \f{d(B)}{2}\rt\rfloor,
      0
    \rt)
    \le
    F_{\fg}
    \lt(
      \lt\lfloor \f{d(A)-\gamma}{2}\rt\rfloor,
      \lt\lfloor \f{d(B)-\gamma}{2}\rt\rfloor,
      \gamma
    \rt) +
    \psi(\gamma).
  \end{equation}
\end{prop}

A complete matching of $\cH_d$ with $\gamma$ cross-edges must have
$\f{d(A)-\gamma}{2}$ $A$-edges and $\f{d(B)-\gamma}{2}$ $B$-edges, respectively.
Moreover, a uniformly random complete matching of $\cH_d$,
conditioned on having $\gamma$ cross-edges, is uniformly distributed
in $\cM\lt(\f{d(A)-\gamma}{2},\f{d(B)-\gamma}{2},\gamma\rt)$.
Thus, $F_{\fg}\lt(\f{d(A)-\gamma}{2},\f{d(B)-\gamma}{2},\gamma\rt)$
is the expected value of $\fg(G)$ for $G$ sampled from $\GG_d$,
conditioned on $G$ having $\gamma$ cross-edges.

Thus, by taking a weighted average of (\ref{eq-graph-interpolation-result})
over $\gamma$ of the correct parity, we get the following result.
\begin{cor}\label{cor-graph-interpolation}
  Let $\fg$ be a graph pseudo-parameter obeying
  (\ref{eq-def-lipschitz-continuity}) and (\ref{eq-def-local-super-additivity}).
  Then,
  \begin{equation}
    F_{\fg}
    \lt(
      \lt\lfloor \f{d(A)}{2}\rt\rfloor,
      \lt\lfloor \f{d(B)}{2}\rt\rfloor,
      0
    \rt)
    \le
    \EE\lt[\fg(\GG_d)\rt] + \psi\lt(|E(\GG_d)|\rt).
  \end{equation}
\end{cor}

\section{Proof of Proposition~\ref{prop-subadd}}

Observe that every matching in $\cM(\alpha+1, \beta, \gamma)$ arises from adding
an $A$-edge to some matching in $\cM(\alpha, \beta, \gamma)$ in $\alpha+1$ ways.
Thus, adding a uniformly random $A$-edge to a matching sampled uniformly
from $\cM(\alpha, \beta, \gamma)$ generates the distribution $\cM(\alpha+1,\beta,\gamma)$.
Analogously, adding a uniformly random $B$ edge or cross-edge generates the
distributions $\cM(\alpha,\beta+1,\gamma)$ and $\cM(\alpha,\beta,\gamma+1)$, respectively.

We will first show that for any $p\ge \f12$,
$HB^{A,B}_{p}$ satisfies the hypotheses
of Corollary~\ref{cor-graph-interpolation}.

\begin{lem}\label{lem-hb-lipschitz}
  For any feasible $(\alpha, \beta, \gamma)$ and $(\alpha', \beta', \gamma')$,
  \begin{equation}
    |
      F_{HB^{A,B}_{p}}(\alpha,\beta,\gamma) -
      F_{HB^{A,B}_{p}}(\alpha',\beta',\gamma')
    |
    \le
    |\alpha-\alpha'|+|\beta-\beta'|+|\gamma-\gamma'|.
  \end{equation}
\end{lem}
\begin{proof}
  It suffices to prove this for the case when
  $(\alpha, \beta, \gamma)$ and $(\alpha', \beta', \gamma')$
  differ by 1 in exactly one coordinate.

  Let $\fm$ be uniformly sampled from $\cM(\alpha,\beta,\gamma)$,
  and $\fm'$ be obtained from $\fm$ by adding a uniformly random $A$-edge.
  Conditioned on any labeling $\Omega$ of the edges of $G[\fm]$,
  the value of $MB^{A,B}(G[\fm](\Omega))$ changes by most 1
  when we add a uniformly random $A$-edge to $\fm$.
  Thus, $|HB^{A,B}_{p}(G[\fm])-HB^{A,B}_{p}(G[\fm'])| \le 1$.
  But, $\fm'$ is uniformly distributed in $\cM(\alpha+1,\beta,\gamma)$,
  so the result follows.

  The argument for $B$-edges and cross-edges is analogous.
\end{proof}

\begin{lem}\label{lem-hb-subadditive}
  If $(\alpha,\beta,\gamma+\delta)$ is feasible, then
  \begin{equation}
    \f12
    \lt(
      F_{HB^{A,B}_{p}}(\alpha+1,\beta,\gamma) +
      F_{HB^{A,B}_{p}}(\alpha,\beta+1,\gamma)
    \rt)
    \le
    F_{HB^{A,B}_{p}}(\alpha,\beta,\gamma+1) + \f{2}{\delta}.
  \end{equation}
\end{lem}
\begin{proof}
  We will prove a stronger claim:
  for any $\fm \in \cM(\alpha,\beta,\gamma)$,
  and any labeling $\Omega$ of $G[\fm]$,
  \begin{align}
    \begin{split}
      & \f12
      \EE\lt[
        MB^{A,B} (G[\fm](\Omega)+e_a)\rt] +
        \f12 \EE\lt[MB^{A,B} (G[\fm](\Omega)+e_b)
      \rt] \\
      & \le
      \EE\lt[
        MB^{A,B} (G[\fm](\Omega)+e_c)
      \rt] +
      \f{2}{\delta},
    \end{split}
  \end{align}
  where $e_a, e_b, e_c$ are
  uniformly random $A$-, $B$-, and cross-edges not in $\fm$,
  labeled $+1$ with probability $p$ and $-1$ with probability $1-p$.
  From this, the desired result follows from
  averaging over all $\fm \in \cM(\alpha,\beta,\gamma)$
  and all labelings $\Omega$ of $G[\fm]$.

  Let $\cC^*$ be the collection of
  maximal $(A,B)$-bisections of $G[\fm](\Omega)$.
  We introduce the equivalence relation $\sim$ on the
  half-edges in $\cH_d$ not paired by $\fm$,
  where $x\sim y$ if the vertices corresponding to $x,y$
  are on the same side of all bisections in $\cC^*$.
  Moreover, we say two equivalence classes are \textit{opposing}
  if their members appear on the opposite side of all bisections of $\cC^*$.

  Let the equivalence classes of $\sim$ be $O_1,P_1,O_2,P_2,\dots,O_k,P_k$,
  where $O_i$ and $P_i$ are opposing.

  If we add a $(+1)$-labeled edge $e^+$ to $G[\fm](\Omega)$,
  its maxcut increases if and only if $e^+$ crosses some cut in $\cC^*$;
  equivalently, the endpoints of $e^+$ must be in different equivalence classes.

  If we add a $(-1)$-labeled edge $e^-$ to $G[\fm](\Omega)$,
  its maxcut decreases if and only if $e^-$ crosses all cuts in $\cC^*$;
  equivalently, the endpoints of $e^-$ must be in opposite equivalence classes.

  Define $o_i^A = |O_i \cap A|$, and define $o_i^B, p_i^A, p_i^B$ analogously.
  Define
  \begin{equation}
    a = \sum_{i=1}^k (o_i^A+p_i^A), \quad\quad b = \sum_{i=1}^k (o_i^B+p_i^B).
  \end{equation}
  It follows that:

  \begin{align}\label{eq-subadd-a}
    \begin{split}
      & \EE\lt[MB^{A,B} (G[\fm](\Omega)+e_a)\rt] - MB^{A,B} (G[\fm](\Omega)) \\
      & =
      p \lt[
        1 - \sum_{i=1}^k \lt(
          \f{o_i^A(o_i^A-1)}{a(a-1)} + \f{p_i^A(p_i^A-1)}{a(a-1)}
        \rt)
      \rt] +
      (1-p) \lt[
        -\sum_{i=1}^k \f{2o_i^Ap_i^A}{a^2}
      \rt] \\
      & = p \lt[
        1 - \sum_{i=1}^k \lt(
          \f{(o_i^A)^2}{a^2} +
          \f{(p_i^A)^2}{a^2} -
          \f{o_i^A (a - o_i^A)}{a^2 (a-1)} -
          \f{p_i^A (a - p_i^A)}{a^2 (a-1)}
        \rt)
      \rt] \\
      & + (1-p) \lt[
        -\sum_{i=1}^k \f{2o_i^Ap_i^A}{a^2}
      \rt].
    \end{split}
  \end{align}
  Analogously,
  \begin{align}\label{eq-subadd-b}
    \begin{split}
      & \EE\lt[MB^{A,B} (G[\fm](\Omega)+e_b)\rt] - MB^{A,B} (G[\fm](\Omega)) \\
      & =
      p \lt[
        1 - \sum_{i=1}^k \lt[
          \f{(o_i^B)^2}{b^2} +
          \f{(p_i^B)^2}{b^2} -
          \f{o_i^B (b - o_i^B)}{b^2 (b-1)} -
          \f{p_i^B (b - p_i^B)}{b^2 (b-1)}
        \rt]
      \rt] \\
      & + (1-p) \lt[
        -\sum_{i=1}^k \f{2o_i^Bp_i^B}{b^2}
      \rt].
    \end{split}
  \end{align}
  and
  \begin{align}\label{eq-subadd-cross}
    \begin{split}
      & \EE\lt[MB^{A,B} (G[\fm](\Omega)+e_c)\rt] - MB^{A,B} (G[\fm](\Omega)) \\
      & =
      p \lt[
        1 - \sum_{i=1}^k \lt[
          \f{o_i^Ao_i^B}{ab} +
          \f{p_i^Ap_i^B}{ab}
        \rt]
      \rt] +
      (1-p) \lt[
        -\sum_{i=1}^k \lt[
          \f{o_i^Ap_i^B}{ab}+\f{p_i^Ao_i^B}{ab}
        \rt]
      \rt].
    \end{split}
  \end{align}

  Equations (\ref{eq-subadd-a}), (\ref{eq-subadd-b}), (\ref{eq-subadd-cross})
  imply:
  \begin{align}
    \begin{split}
      & \f12 \EE\lt[MB^{A,B} (G[\fm](\Omega)+e_a)\rt] +
      \f12 \EE\lt[MB^{A,B} (G[\fm](\Omega)+e_b)\rt] \\
      & -
      \EE\lt[MB^{A,B} (G[\fm](\Omega)+e_c)\rt] \\
      &=
      -\f12p \sum_{i=1}^k \lt[
        \f{(o_i^A)^2}{a^2} +
        \f{(p_i^A)^2}{a^2} +
        \f{(o_i^B)^2}{b^2} +
        \f{(p_i^B)^2}{b^2} -
        \f{2o_i^Ao_i^B}{ab} -
        \f{2p_i^Ap_i^B}{ab}
      \rt] \\
      &
      - \f12(1-p) \sum_{i=1}^k \lt[
        \f{2o_i^Ap_i^A}{a^2} +
        \f{2o_i^Bp_i^B}{b^2} -
        \f{2o_i^Ap_i^B}{ab} -
        \f{2p_i^Ao_i^B}{ab}
      \rt]  \\
      &
      + \f12p \sum_{i=1}^k \lt[
        \f{o_i^A (a - o_i^A)}{a^2 (a-1)} +
        \f{p_i^A (a - p_i^A)}{a^2 (a-1)} +
        \f{o_i^B (b - o_i^B)}{b^2 (b-1)} +
        \f{p_i^B (b - p_i^B)}{b^2 (b-1)}
      \rt].
    \end{split}
  \end{align}
  The first main observation is that
  \begin{equation}
    \sum_{i=1}^k \lt[
      \f{o_i^A (a - o_i^A)}{a^2 (a-1)} +
      \f{p_i^A (a - p_i^A)}{a^2 (a-1)}
    \rt]
    \le
    \sum_{i=1}^k \lt[
      \f{o_i^A}{a (a-1)} +
      \f{p_i^A}{a (a-1)}
    \rt]
    =
    \f{1}{a-1}
    \le
    \f{2}{\delta},
  \end{equation}
  and analogously
  \begin{equation}
    \sum_{i=1}^k \lt[
      \f{o_i^B (b - o_i^B)}{b^2 (b-1)} +
      \f{p_i^B (b - p_i^B)}{b^2 (b-1)}
    \rt] \le \f{2}{\delta}.
  \end{equation}
  So,
  \begin{align}
    \begin{split}
      &\f12p \sum_{i=1}^k \lt[
        \f{o_i^A (a - o_i^A)}{a^2 (a-1)} +
        \f{p_i^A (a - p_i^A)}{a^2 (a-1)} +
        \f{o_i^B (b - o_i^B)}{b^2 (b-1)} +
        \f{p_i^B (b - p_i^B)}{b^2 (b-1)}
      \rt]
      \le
      \f{2p}{\delta} \\
      &\le
      \f{2}{\delta}.
    \end{split}
  \end{align}
  So, it remains to show
  \begin{align}\label{eq-desired-ineq}
    \begin{split}
      &
      -\f12p \sum_{i=1}^k \lt[
        \f{(o_i^A)^2}{a^2} +
        \f{(p_i^A)^2}{a^2} +
        \f{(o_i^B)^2}{b^2} +
        \f{(p_i^B)^2}{b^2} -
        \f{2o_i^Ao_i^B}{ab} -
        \f{2p_i^Ap_i^B}{ab}
      \rt] \\
      &
      - \f12(1-p) \sum_{i=1}^k \lt[
        \f{2o_i^Ap_i^A}{a^2} +
        \f{2o_i^Bp_i^B}{b^2} -
        \f{2o_i^Ap_i^B}{ab} -
        \f{2p_i^Ao_i^B}{ab}
      \rt]
      \le 0. \\
    \end{split}
  \end{align}
  The second main observation is that the
  left-hand side is a linear function of $p$,
  so verifying (\ref{eq-desired-ineq}) at $p=1$ and $p=\f12$ is sufficient.
  At $p=1$, (\ref{eq-desired-ineq}) follows from:
  \begin{equation}
    -\f12 \sum_{i=1}^k \lt[
      \lt(\f{o_i^A}{a}-\f{o_i^B}{b}\rt)^2 +
      \lt(\f{p_i^A}{a}-\f{p_i^B}{b}\rt)^2
    \rt] \le 0.
  \end{equation}
  At $p=\f12$, (\ref{eq-desired-ineq}) follows from:
  \begin{equation}
    -\f14 \sum_{i=1}^k \lt[
      \f{o_i^A}{a} +
      \f{p_i^A}{a} -
      \f{o_i^B}{b} -
      \f{p_i^B}{b}
    \rt]^2 \le 0.
  \end{equation}
\end{proof}

We are now ready to prove Proposition~\ref{prop-subadd}.
\begin{proof}[Proof of Proposition~\ref{prop-subadd}]
  By Propositions~\ref{lem-hb-lipschitz} and \ref{lem-hb-subadditive},
  $HB^{A,B}_{p}$ satisfies the hypotheses of
  Corollary~\ref{cor-graph-interpolation}.
  Thus,
  \begin{equation}
    F_{HB^{A,B}_{p}}
    \lt(
      \lt\lfloor \f{d(A)}{2}\rt\rfloor, \lt\lfloor \f{d(B)}{2}\rt\rfloor, 0
    \rt)
    \le
    \EE\lt[HB^{A,B}_{p}(\GG_d)\rt] +
    \psi\lt(|E(\GG_d)|\rt).
  \end{equation}

  The graphs arising from
  $\cM\lt(
    \lt\lfloor \f{d(A)}{2}\rt\rfloor, \lt\lfloor \f{d(B)}{2}\rt\rfloor, 0
  \rt)$
  have no cross-edges.
  So, an optimal $(A,B)$-bisection of these graphs is the sum of
  an optimal bisection of $A$ and an optimal bisection of $B$.
  Thus,
  \begin{equation}
    F_{HB^{A,B}_{p}}
    \lt(
      \lt\lfloor \f{d(A)}{2}\rt\rfloor, \lt\lfloor \f{d(B)}{2}\rt\rfloor, 0
    \rt)
    =
    \EE\lt[HB_p(\GG_{d\uparrow A})\rt] +
    \EE\lt[HB_p(\GG_{d\uparrow B})\rt],
  \end{equation}
  as desired.
\end{proof}

\section{Proof of Theorem~\ref{thm-hybrid}}

On the space of probability measures on $\NN$ with finite mean,
define the \textit{Wasserstein distance}
\begin{equation}
  \cW(\mu, \mu')
  =
  \sum_{i=1}^\infty \lt| \sum_{k=i}^\infty \lt(\mu(k) - \mu'(k\rt)\rt|.
\end{equation}

We will use the following result from \cite{Sal}.
\begin{prop}\label{prop-mu-lipschitz}
  Let $\ff$ be a $1$-Lipschitz graph parameter.
  For any $d,d':[n]\functo \NN$,
  \begin{equation}
    \lt|
      \f{1}{n}\EE\lt[\ff(\GG_{d})\rt] -
      \f{1}{n}\EE\lt[\ff(\GG_{d'})\rt]
    \rt|
    \le
    2 \cW\lt(
      \f1n \sum_{i=1}^n \delta_{d(i)},
      \f1n \sum_{i=1}^n \delta_{d'(i)}
    \rt).
  \end{equation}
\end{prop}

Throughout this proof,
let $\GG_{\mu, n}^{IID}$ denote the random graph on $n$ vertices,
where each vertex's degree is sampled i.i.d. from the distribution $\mu$.

\begin{proof}[Proof of Theorem~\ref{thm-hybrid}]
  Proposition~\ref{prop-subadd} immediately implies the subadditivity of $HB_p$:
  for all partitions $A,B$ of $[n]$, and all $d:[n]\functo \NN$,
  \begin{equation}
    \EE\lt[HB_p(\GG_{d\uparrow A})\rt] +
    \EE\lt[HB_p(\GG_{d\uparrow B})\rt]
    \le
    \EE\lt[HB_p(\GG_d)\rt] +
    \psi\lt(|E(\GG_d)|\rt).
  \end{equation}
  Fix a distribution $\mu$ on $\NN$ with finite mean.
  By averaging the above inequality over $d$
  whose values are sampled i.i.d. from $\mu$, we have
  \begin{equation}
    \EE\lt[
      HB_p(\GG_{\mu, |A|}^{IID})\rt] +
      \EE\lt[HB_p(\GG_{\mu, |B|}^{IID})
    \rt]
    \le
    \EE\lt[HB_p(\GG_{\mu, n}^{IID})\rt] +
    \psi\lt(\f12 \mean{\mu} n\rt),
  \end{equation}
  where we have used Jensen's Inequality on the concavity of $\psi$.

  Note that $\psi(\f12 \mean{\mu} n) = o\lt(n^{2/3}\rt)$.
  By Fekete's Subadditivity Lemma, this implies that the scaling limit
  \begin{equation}
    \lim_{n\approach \infty}\f1n \EE\lt[HB_p(\GG_{\mu, n}^{IID})\rt]
  \end{equation}
  exists.
  Let this limit equal $\Psi(\mu)$.

  Since $HB_p$ is Lipschitz, Proposition~\ref{prop-mu-lipschitz} applies.
  By setting $d=d_n$, sampling each $d'(1),\dots,d'(n)$ uniformly from $\mu$,
  and taking the limit as $n\approach\infty$, we get
  \begin{equation}
    \lt|
      \f{1}{n}\EE\lt[\ff(\GG_{d_n})\rt] -
      \f{1}{n}\EE\lt[\ff(\GG^{IID}_{\mu, n})\rt]
    \rt|
    \approach 0.
  \end{equation}
  Therefore,
  \begin{equation}
  \lim_{n\approach \infty } \f{1}{n}\EE\lt[\ff(\GG_{d_n})\rt] = \Psi(\mu)
  \end{equation}
  as well.
  Moreover, as $HB_p$ is Lipschitz,
  Azuma-Hoeffding's inequality implies the concentration inequality
  \begin{equation}
  \PP\lt[
    \lt|HB_p (\GG_d)- \EE\lt[HB_p(\GG_d)\rt]\rt| \ge \eps
  \rt]
  \le
  \exp\lt(
    -\f{\eps^2}{4 \sum_{i=1}^n d(i)}
  \rt).
  \end{equation}
  Since $d$ converges in distribution to $\mu$,
  the Borel-Cantelli Lemma implies the almost-sure convergence
  \begin{equation}
  \lt|
    \f{1}{n}HB_p(\GG_{d_n}) - \f{1}{n}\EE\lt[HB_p(\GG_{d_n})\rt]
  \rt| \approach 0.
  \end{equation}
  The result follows.
\end{proof}


\begin{thebibliography}{1}

\bibitem{BGT}
M.~Bayati, D.~Gamarnik, and P.~Tetali.
\newblock Combinatorial appproach to the interpolation method and scaling
  limits in sparse random graphs.
\newblock {\em The Annals of Probability}, 41(6):4080--4115, 2013.

\bibitem{FrLe}
S.~Franz and M.~Leone.
\newblock Replica bounds for optimization problems and diluted spin systems.
\newblock {\em Journal of Statistical Physics}, 111(3-4):535--564, 2003.

\bibitem{FLT}
S.~Franz, M.~Leone, and F.~Toninelli.
\newblock Replica bounds for diluted non-poissonian spin systems.
\newblock {\em Journal of Physics A: Mathematical and General},
  36(43):10967--10985, 2003.

\bibitem{Gam}
D.~Gamarnik.
\newblock Right-convergence of sparse random graphs.
\newblock {\em Probability Theory and Related Fields}, 160(1-2):253--278, 2014.

\bibitem{GuTo}
F.~Guerra and F.~Toninelli.
\newblock The thermodynamic limit in mean field spin glass models.
\newblock {\em Communications in Mathematical Physics}, 230(1):71--79, 2002.

\bibitem{Aim}
American~Institute of~Mathematics.
\newblock Phase transitions in randomized computational problems, available at
  http://aimpl.org/phaserandom.

\bibitem{PaTa}
D.~Panchenko and M.~Talagrand.
\newblock Bounds for diluted mean-fields spin glass models.
\newblock {\em Probability Theory and Related Fields}, 130:319--336, 2004.

\bibitem{Sal}
J.~Salez.
\newblock The interpolation method for random graphs with prescribed degrees.
\newblock {\em Combinatorics, Probability, and Computing}, 25:436--447, 2015.

\bibitem{ZdBo}
L.~Zdeborov\'a and S.~Boetther.
\newblock Conjecture on the maximum cut and bisetion width in random regular
  graphs.
\newblock {\em Journal of Statistical Mechanics}, page P02020, 2010.

\end{thebibliography}
\end{document}